\documentclass{article}

\usepackage{amsmath}
\usepackage{amsthm}
\usepackage{mathrsfs}
\usepackage{latexsym}
\usepackage{amssymb}
\usepackage{amscd}
\usepackage[dvips]{graphics}
\usepackage[all]{xy}
\usepackage{textcomp}
\usepackage{verbatim}

\DeclareSymbolFont{AMSb}{U}{msb}{m}{n}
\DeclareMathSymbol{\N}{\mathbin}{AMSb}{"4E}
\DeclareMathSymbol{\Z}{\mathbin}{AMSb}{"5A}
\DeclareMathSymbol{\R}{\mathbin}{AMSb}{"52}
\DeclareMathSymbol{\Q}{\mathbin}{AMSb}{"51}
\DeclareMathSymbol{\I}{\mathbin}{AMSb}{"49}
\DeclareMathSymbol{\C}{\mathbin}{AMSb}{"43}
\numberwithin{equation}{section}

\newcommand{\iso}{\cong}

\newcommand{\Image}{\textnormal{Im}}

\newcommand{\Oplus}{\bigoplus}
\newcommand{\lmlex}{\textnormal{lm}_{\textnormal{lex}}}
\newcommand{\Glm}{G\textnormal{-lm}}
\newcommand{\Red}{\textnormal{Red}}
\newcommand{\stab}{\textnormal{stab}}

\newtheorem{defn}[equation]{Def{i}nition}

\newtheorem{lemma}[equation]{Lemma}
\newtheorem{thm}[equation]{Theorem}
\newtheorem{corol}[equation]{Corollary}
\newtheorem{example}[equation]{Example}

\makeatletter
\newtheorem*{rep@theorem}{\rep@title}
\newcommand{\newreptheorem}[2]{%
\newenvironment{rep#1}[1]{%
 \def\rep@title{#2 \ref{##1}}%
 \begin{rep@theorem}}%
 {\end{rep@theorem}}}
\makeatother

\newreptheorem{theorem}{Theorem}
\newreptheorem{lemma}{Lemma}
\newreptheorem{corol}{Corollary}

\begin{document}

\markboth{Robert Mckemey}
{Structure Theorems for the Symmetric Group}

%
%

\title{STRUCTURE THEOREMS FOR THE SYMMETRIC GROUPS ACTING ON ITS NATURAL MODULE
}

\author{ROBERT MCKEMEY}


\maketitle

\begin{abstract}
This paper gives an explicit structure theorem for the symmetric group acting on the symmetric algebra of its natural module.
Let $G$ be the symmetric group on $x_1, \dots , x_n$ and let $d_i$ be the $i^{\text{th}}$ elementary symmetric polynomial in the $x_i$'s.
We show that if we take monomial representations discussed in \cite[Section 3]{Kemper} to be the modules $V_I$,
then we have an isomorphism of $kG$-modules $k[x_1, \dots , x_n] \iso \Oplus_{\{n\} \subseteq I \subseteq [n]} k[d_I] \otimes_k V_I$.
\end{abstract}



This paper gives a structure theorem for the symmetric group, $G$, acting on its natural module,
which gives us a $kG$-decomposition of the graded components of $S=k[x_1, \dots, x_n]$,
where $k$ is a unital ring such that $ab=0$ implies $a=0$ or $b=0$ for $a,b \in k$.
Which is to say, for $d_1, \dots, d_n$ the elementary symmetric polynomials in $x_1, \dots, x_n$,
we give $kG$-submodules of $S$, $V_I$ for $I \subseteq \{1, \dots, n\}$ $n \in I$,
such that the multiplication map $\Oplus_{\{n\} \subseteq I \subseteq \{1, \dots, n\}} k[d_I] \otimes_k V_I \to S$
is a $kG$-isomorphism.

In fact the monomial representations discussed in \cite[Section 3]{Kemper} maybe be taken as the modules, $V_I$,
occurring in a structure theorem.
Many of the intermediate steps will be similar to those from \cite{Kemper},
but the fact that we get a structure theorem is new
as is the observation that we may use $e_I$, rather than the $e_I'$ used by Kemper.
Note that although the ring $k$ need not be commutative,
we require that $ax_i=x_ia$ for $i=1, \dots, n$ and for all $a \in k$.


It will turn out that in this example of a structure theorem 
all $V_I$ with $n \not \in I$ are zero,
this was also true for the upper triangular structure theorem.

For more information on structure theorems see
\cite{Symonds mod str 1}, \cite{Symonds mod str 2}, \cite{Symonds mod str 3}, \cite{Symonds mod str 0}
and \cite{Symonds StrThm}.
A more verbose exposition of this material and additional examples of structure theorems can be found in \cite{MyThesis}.

\section{Definition and Results in the Literature}

Let $k$ be a unital ring such that $ab=0$ implies $a=0$ or $b=0$,
which need not be commutative.
Let $R=k[d_1, \dots , d_n]$ be the $\N$ graded polynomial $k$-algebra in the
indeterminants $d_1, \dots , d_n$, with $\deg(d_i)>0$ but not necessarily with $\deg(d_1)=1$.
Let $G$ be any finite group and let $S$ be a finitely generated $\Z$-graded $RG$-module.

\begin{defn}
With notation as above,
a \emph{Structure Theorem} for $S$ over $RG$ is a set of finitely generated $kG$-submodules, $X_I \subseteq S$,
one for each $I \subseteq \{1, \dots , n \}$,
such that the map:
\begin{align*}
\phi: \Oplus_{I \subseteq \{1, \dots , n\} } k[d_i | i \in I] \otimes_k X_I & \to S \\
\phi: d \otimes_k x & \mapsto dx
\end{align*}
is an isomorphism of $kG$-modules.
\end{defn}

Note that the map $\phi$ is split over $kG$, as it is a $kG$-isomorphism.
As the module being mapped from is not an $R$-module, it cannot hope to be an $R$-map,
however the following lemma is straightforward.

\begin{lemma}
For each component of the sum, the map:
\begin{align*}
\phi_I:k[d_i | i \in I] \otimes_k X_I & \to S \\
\phi_I : d \otimes_k x & \mapsto dx
\end{align*}
is a $k[d_i | i \in I]G$-homomorphism.
\end{lemma}

If we insist that $k$ is a field, then we know that a structure theorem exists for the symmetric group acting on its natural module by the following arguments.

\begin{thm}[Symonds 2006]\cite{Symonds StrThm} \label{thm: symonds strthm}
Let $k$ be a field and let $R=k[d_1, \dots , d_n]$ be the graded polynomial ring with $\deg(d_i)>0$ for all $i$,
let $G$ be a finite group graded in degree 0
and let $S$ be a finitely generated $\Z$-graded $RG$-module.
A structure theorem for $S$ exists exactly when
only finitely many isomorphism classes of indecomposable $kG$-modules occur as summands of $S$.
\end{thm}

Note that since we insisted that the $X_I$ are finitely generated it is not the case that
every $S$ trivially has a structure theorem given by $X_{\emptyset} = S|_{kG}$.

Let $S=k[x_1, \dots , x_n]$ be a polynomial ring in $n$ variables graded in degree 1.
With respect to the basis $x_1, \dots , x_n$ of the degree 1 component of $S$,
let $P$ denote a finite subgroup of the upper triangular of matrices with 1's on the diagonal.

\begin{thm} [Karagueuzian and Symonds 2007]
\cite[Theorem 1.1]{Symonds mod str 2} \label{thm: symonds mod str 2}
For $k$ a finite field, $S$ and $P$ as immediately above and $R \subset S^P$ a particular Noether normalization of $S^P$,
 $S$ has a structure theorem over $RP$.
\end{thm}

Any group acting on $S$ with grading preserving algebra automorphisms is defined by its action on the degree 1 component of $S$.
%
Let $P$ be any Sylow-$p$-subgroup of $G$.
It can be shown that we may chose a basis of the degree 1 component of $S$
such that the elements of $P$ are represented by upper triangular matrices with $1$'s on the diagonal.
%
A similar argument to Theorem \ref{thm: symonds mod str 2} (found in the proof of \cite[Corollary 4.2]{Symonds cech})
tells us that $S$ has a structure theorem over $RP$.
Since $P$ is a Sylow-$p$-subgroup of $G$, this tells us that $S$ has finitely many isomorphism classes of indecomposable $kG$-summands.
Hence $S$ has a structure theorem over $RG$.
All together this shows:



\begin{corol}\label{corol: k[d] has str}
Let $k$ be a field of characteristic $p$.
For $S=k[x_1, \dots , x_n]$, with $\deg(x_i)=1$ for $i=1, \dots , n$, $G$ a finite group of grading preserving algebra automorphisms
and $R \subseteq S^G$ a polynomial ring such that $S$ is a finite $R$-module,
$S$ has a structure theorem over $RG$. cf. \cite[Corollary 1.2]{Symonds StrThm}.
\end{corol}

\section{Notation}   \label{sec: str for sym}

We now fix notation which we will use for the rest of the paper.

Take $k$ to be any unital ring such that $ab=0$ implies $a=0$ or $b=0$.
Fix an $n \in \N_{>0}$, let $S=k[x_1, \dots , x_n]$ and let $G= \textnormal{Sym}(x_1, \dots , x_n)$
be the symmetric group on the variables $x_1, \dots ,x_n$.
Let $d_i$ be the $i^{\text{th}}$ elementary symmetric polynomial in $x_1, \dots ,x_n$
e.g. $d_1 = x_1 + \dots + x_n$, $d_n = x_1x_2 \dots x_n$ and for $i \in \{1, \dots , n\}$:
\begin{align*}
d_i &= \sum_{g \in G/ \stab_G(x_1 \dots x_i)} g (x_1 \dots x_i).
\end{align*}
Let $R=k[d_1, \dots , d_n]$.
It is well known that the $d_i$ are algebraically independent
(a result sometimes called the fundamental theorem of symmetric polynomials),
 so $R$ is a polynomial $k$-algebra.

Note that $\stab_G(x_1 \dots x_i)$ is the stabilizer of the monomial $x_1 \dots x_i$,
which is the same as the stabilizer of the set $\{x_1 , \dots , x_i\}$.
Elements of this group are made up of a permutation of $x_1, \dots , x_i$ and a permutation of $x_{i+1}, \dots , x_n$.

For any $m \in \N$ let $[m]=\{1,2, \dots , m\}$.
For $I = \{i_1, \dots ,i_m \} \subseteq [n]$, let
$\tilde{d_I}=d_{i_1} \dots d_{i_{m}} = \prod_{i \in I} d_i$, we let $\tilde{d_{\emptyset}}=1$.
Let $d_I$ denote the set $\{ d_i | i \in I  \}$ and $k[d_I]$ denote the polynomial ring $k[d_i|i \in I]$.

Let $\lmlex$ denote the leading monomial in the usual lexicographical ordering on monomials in $x_1, \dots ,x_n$.
For $I = \{i_1, \dots i_m \} \subseteq [n]$ with $n \in I$, set $e_I' = \lmlex( \tilde{d_I})/d_n$ and let $V_I'$ be the $kG$-module generated by $e_I'$.
An element of $R$ of the form $d_{i_1}^{t_1}\dots t_{i_m}^{t_m}$
we will call a $d_I$-monomial, and if $I=[n]$ we may shorten this to a $d$-monomial.
Likewise $x_I$ and $x$-monomials, are elements of $S$ of the form $x_{i_1}^{t_1} \dots x_{i_m}^{t_m}$,
with  $i_j \in I$ and $[n]$ respectively.

Note that if we defined $e_I'' = \lmlex( \tilde{d_I})$ and $V_I'' = \langle e_I'' \rangle$ for any $I \subseteq [n]$,
then for any $I$ with $n \not \in I$ we would have $V_I'' \iso V_{I \cup \{n\}}''$,
and for any $I$ with $n \in I$ we would have $V_I' \iso V_I''$.
So no new isomorphism classes of module occur for $V_I''$ with $n \not \in I$.
In both cases the isomorphism is given by multiplication by $d_n$.

This notation is summarized in the top part of table below,
for now ignore the bottom two rows as $\Glm$ has not yet been defined.

\begin{tabular}{|l|l|}
\hline
$k$ & unital ring such that $ab=0 \implies a=0$ or $b=0$ for all $a,b \in k$ \\
$S$ & $k[x_1, \dots , x_n]$ \\
$G$ & Sym$(x_1, \dots , x_n)$ \\
$d_i$ & $i^{\text{th}}$ elementary symmetric polynomial in $x_1, \dots ,x_n$ \\
& i.e. $d_i = \sum_{g \in G/stab_G(x_1 \dots x_i)} g x_1 \dots x_i$ \\
$R$ & $k[d_1, \dots , d_n]$ \\
$[m]$ &  $\{1,2, \dots , m\}$ \\
$I$ & $I \subseteq [n], I = \{i_1 , \dots , i_{|I|} \}$ \\
$\tilde{d_I}$ & $d_{i_1} \dots d_{i_{|I|}}= \prod_{i \in I} d_i$ for $\{n\}\subseteq I \subseteq [n]$ \\
$d_I$ & $\{ d_i | i \in I\}$\\
$e_I'$ & $\lmlex( \tilde{d_I})/d_n$ for $\{ n \} \subseteq I \subseteq [n]$\\
$V_I'$ & the $kG$-module generated by $e_I'$ for $\{ n \} \subseteq I \subseteq [n]$\\
\hline
$e_I$ & an element of $S$ such that $\Glm(e_I) = \{ \lmlex( \tilde{d_I})/d_n \}$, \\
&  stab$_{G}(e_I)=$stab$_G(\lmlex (e_I) )$ for $\{ n \} \subseteq I \subseteq [n]$ \\
& and the coefficient of the $\succ$-leading monomial is a unit (e.g. $e_I =e_I'$)\\
$V_I$ & the $kG$-module generated by $e_I$ for $\{ n \} \subseteq I \subseteq [n]$\\
\hline
\end{tabular} \\

The result we are aiming for is:
\begin{thm}
With notation as above, we have a structure theorem:
$$
S \iso \Oplus_{\{ n \} \subseteq I \subseteq [n]} k[d_I] \otimes_k V_I'
$$
Where the map from right to left is the $kG$-homomorphism $d \otimes_k v \mapsto dv$
\end{thm}

This is an immediate corollary of Theorem \ref{thm: str thm},
where $e_I'$ and $V_I'$ are replaced by $e_I$ and $V_I$.

Using $e_I$, rather than $e_I'$, does make the notation a little more messy
but being able to use $e_I$  allows more flexibility. 
It may also be useful for considering localizations of $S$.
For example, assume the $e_I$ version of the theorem holds and fix $r \in [n]$,
then the following choices for $e_I$ are allowed:
$$
e_I = \left\{
\begin{array}{cl}
e_I' & \textnormal{if } r\not \in I \\
d_r e_{I - \{r\}}' & \textnormal{if } r \in I
\end{array} \right.
$$
If $r \not \in I$ then $d_rV_I \subseteq \Image (1_R \otimes_k V_{I \cup \{r\}} )$.
On the other hand if $r \in I$, since the theorem holds, we have $d_rV_I = \Image (d_r \otimes_k V_I)$.
So for all $I \subseteq [n]$ with $n \in I$, we have $d_r V_I \subset \Image\left( k[d_{I \cup \{r\}}] \otimes_k V_{I \cup \{r\}} \right)$.
This tells you that for $S_{d_r}$, the localization of $S$ by $d_r$,
we have a split isomorphism of $kG$-modules:
$$
S_{d_r} \iso \Oplus_{\{I \subseteq [n] | r,n \in I\}} k[d_I][d_r^{-1}] \otimes_k V_I
$$
where the isomorphism from right to left is given by multiplication.

The two main tools we use are the $\succ$-leading monomials and the \emph{reduced form}.

\section{Leading Monomials}

The following definitions are similar to \cite[Section 3 Definition 13]{Kemper}.

\begin{defn}[$\succ ,\succcurlyeq , \approx, M(-),\Glm$]
For two $x$-monomials, $y,z \in S$, pick $g,h \in G$ such that
$gy \geq_{\textnormal{lex}} g'y$ for all $g' \in G$ and
$hz \geq_{\textnormal{lex}} h'z$ for all $h' \in G$,
we say that $y \succcurlyeq z$ if
$gy \geq_{\textnormal{lex}} hz$,
otherwise $y \prec z$.

We say that
$x \approx y$ if $x \preccurlyeq y$ and $y \preccurlyeq x$, i.e. if there exists $g,h \in G$ such that $gx=hy$.

For $u \in S$, define $M(u)$ to be the set of $x$-monomials occurring in $u$ (with non-zero coefficient) and define
$$
\Glm(u) := \{ x \in M(u) | x  \succcurlyeq y  \textnormal{ for all }  y\in M(u) \}
$$

For a set $X$ such that $x \approx y$ for all $x,y \in X$,
write $X \approx m$ if $\forall x \in X, x \approx m$.

Note that $\forall x,y \in \Glm(u), x \approx y$, so $\Glm(u) \approx m$ makes sense.
\end{defn}

Note the distinction between $\Glm(u) \approx m$ and $\Glm(u) = \{ m \}$ for $u,m \in S$, $m$ an $x$-monomial.
The former says that the leading monomials of $u$ in the $\succcurlyeq$ ordering are all equal to $gm$ for some $g \in G$.
The latter says that there is exactly one $n \in M(u)$
which is maximal in the $\succcurlyeq$ ordering and this $n$ is equal to $m$.

Let $e_I$ and $V_I$ be as in the box from Section \ref{sec: str for sym},
i.e. for $I \subseteq [n]$ with $n \in I$: $e_I$ is an element of $S$ such that
$\Glm(e_I) = \{\lmlex (\tilde{d_I})/d_n \}=\{e_I'\}$,
$\stab_G(e_I) = \stab_G(e_I')$ and the coefficient of the $\succ$-leading monomial is a unit;
$V_I$ is the module $kGe_I$.

The condition that $\Glm(e_I) = \{ e_I' \}$ could be relaxed to $\Glm(e_I) = \{ g \cdotp e_I' \}$,
or we could say $\Glm(e_I) \approx \{ e_I' \}$ and $|\Glm(e_I)|=1$.
We gain no benefit from this as the next lemma tells us that for such an $e_I$ we would have $\Glm(g^{-1}e_I) = \{e_I' \}$,
so the $V_I$ obtained in this way are the same.
So we insist that $\Glm(e_I) = \{e_I' \}$.

\begin{lemma}\label{lem: Glm}
Let $d$ be a $d$-monomial considered as an element of $S$ and $u,v,w$ be any elements of $S$ then:
\begin{enumerate}
\item $\lmlex (uv) = \lmlex (u)\lmlex (v)$
\item $\Glm(d) \approx \lmlex(d)$
\item $\Glm(d) = \{g \cdotp \lmlex(d) | g \in G \}$
\item For any $g \in G$ we have $\Glm(u) \approx \Glm(gu)$.
\item Let $m$ be an $x$-monomial with $m \in \Glm(u)$, $\Glm(u) = Gm \cap M(u)$
\item For $\Glm(u)$ and $\Glm(v)$ disjoint, $\Glm (u+v) \subseteq \Glm (u) \cup \Glm(v)$,
in particular $\Glm(u+v) \approx \Glm(u)$ or $\Glm(u+v) \approx\Glm(v)$.
\item $\Glm(gu) = g\Glm(u)$ for all $g \in G$.
\end{enumerate}
\end{lemma}
\begin{proof}
(1) follows from $a \geq_{\textnormal{lex}} b \implies ac \geq_{\textnormal{lex}} bc$ for $x$-monomials $a,b,c$.

(2) and (3) are because $d$ is a $d$-monomial.

(4) is because $m \approx gm$ and $m \succ n \iff gm \succ gn$, for $x$-monomial $m,n$.

(5) if $n \in \Glm(u)$, then  $n \succcurlyeq  m'$ for all $m' \in M(u)$, so in particular $n \succcurlyeq m$.
We already know  $m \succcurlyeq n$, so $m \approx n$, i.e. $n\in Gm$.
Clearly $\Glm(u) \subseteq M(u)$, so $\Glm(x) \subseteq Gm \cap M(x)$.

Conversely, if $n \in Gm \cap M(u)$, then $n \approx m$ and $m \succeq m'$ for all $m' \in M(u)$.
So $n \succeq m'$, for all $m' \in M(u)$, and $n \in M(u)$, so $n \in \Glm(u)$.

(6) for $m \in \Glm(u)$, $n \in \Glm(v)$, without loss of generality let $m \succcurlyeq n$.
Then $m \in M(u+v)$, as $m \not \in M(v)$, and $m \succcurlyeq m'$ for all $m' \in M(u+v)$.

(7) suppose $m \in \Glm(u)$, then
$gm \in Gm \cap M(gu)$ so $gm \in \Glm(gu)$ by part (5).
This shows that $g\Glm(u) \subseteq \Glm(gu)$ and $g^{-1}\Glm(gu) \subseteq \Glm(g^{-1}gu)$.
%
\end{proof}

\begin{lemma} \label{lem: Glm(u)=lm(e_I)}
For $e_I$ and $V_I$ as in the box and $u \in V_I$ with $u \neq 0$,
we have $\Glm(u) \approx e_I'$
\end{lemma}
\begin{proof}
As $V_I = kGe_I$, for $T$ a transversal of $\stab_G(e_I)$ in $G$,
any non-zero element of $V_I$ may be expressed uniquely as a sum:
$$
\sum_{g \in T} \lambda_g ge_I,
$$
for $\lambda_g \in k$, with at least one $\lambda_g$ non-zero.

Since $\stab_G(e_I) = \stab_G(e_I')$,
the $T$ we chose above is a transversal of $\stab_G(e_I')$ in $G$.
By definition $\Glm(e_I) = \{ e_I' \}$,
hence by Lemma \ref{lem: Glm}(7), $\Glm(ge_I)=g\Glm(e_I) = \{g \cdot e_I' \}$ .
So for $g,h \in T$ we have that $\Glm(g e_I)$ and $\Glm(h e_I)$ are disjoint when $g \neq h$.

By repeated application of Lemma \ref{lem: Glm} (6),
for $\lambda_g \in k$ with at least one of the $\lambda_g \neq 0$ we have:
$$
\Glm \left( \sum_{g \in T}  \lambda_g ge_I \right) \approx \Glm(g' e_I),
$$
for some $g' \in T$ and by Lemma \ref{lem: Glm}(4), $\Glm(g' e_I) \approx e_I'$ for all $g' \in T$.
\end{proof}

\begin{lemma}\label{lem: Glm(du) = lm(de_I)}
For $e_I$ and $V_I$ as in the box, $d$ a $d_I$-monomial and $u \in S$,
if $\Glm(u)\approx e_I'$,  then $\Glm(du)\approx \lmlex(d)e_I'$.

In particular, for $u \in V_I-\{0\}$ we have: $\Glm(du) \approx \lmlex(d)e_I'$.
\end{lemma}
\begin{proof}
Take $m \in \Glm(u)$,  there exists a $g \in G$ such that $gm= e_I'$.
By Lemma \ref{lem: Glm}(7) $g\Glm(u)=\Glm(gu)$, and by Lemma \ref{lem: Glm}(4),
$\Glm(gu) \approx \Glm(u)$.
So we may assume $e_I' \in \Glm(u)$ and $\lmlex(u)=e_I'$.
Hence $\lmlex(d)e_I' = \lmlex(du)$ by Lemma \ref{lem: Glm}(1),
in particular $\lmlex(d)e_I'\in M(du)$.
So it is sufficient to show that $\lmlex(d)e_I' \succeq n$ for all $n \in M(du)$.

For $d = d_1^{t_1}\dots d_n^{t_n}$:
$$
M(du) = \left\{ \left. \left(\prod_{i = 1}^n \prod_{j=1}^{t_i} g_{i,j}\lmlex(d_i)\right) a \right| a \in M(u), g_{i,j} \in G\right\}
$$
That $\Glm(u) \approx e_I'$,
implies that for all $h \in G$ and all $a \in M(u)$, we have
$e_I' \geq_{\textnormal{lex}} ha$.
Clearly $\lmlex(d_i) \geq_{\textnormal{lex}} g_{i,j}\lmlex(d_i)$,
so $\lmlex(d)e_I' \geq_{\textnormal{lex}} hn$ for all $n \in M(du)$.

The ``in particular'' statement follows from Lemma \ref{lem: Glm(u)=lm(e_I)}.
\end{proof}

\section{Reduced Form}

The following definition is equivalent to \cite[Section 3 Definition 10]{Kemper},
where it is described as a generalization of G\"{o}bel's concept of ``special'' terms.

\begin{defn}
For an $x$-monomial $m \in S, m= x_1^{r_1} \dots x_n^{r_n}$,
 the \emph{reduced form} of $m$,
$\Red(m)$, is the $x$-monomial $x_1^{r'_1} \dots x_n^{r'_n}$ where:
\begin{itemize}
\item $| \{ r'_i | i = 1, \dots , n\} | = | \{ r_i | i = 1, \dots , n\} |= a \leq n$
\item $\{ r'_i | i = 1, \dots , n\} = \{0, 1, \dots , a-1 \} $
\item $r'_i < r'_j \iff r_i < r_j$ for all $i,j$.
\end{itemize}
We say that an $x$-monomial, $m$, is \emph{in reduced form} if $m = \Red(m)$.
\end{defn}

Note that for every $x$-monomial in $S$, $m$, there exists a $g \in G$ such that $gm$ is the leading $x$-monomial of some $d$-monomial.
This is simply the observation that every $x$-monomial $m=x_1^{m_1} \dots x_n^{m_n}$ with $m_1\geq m_2 \geq \dots \geq m_n$
can be written as $x_1^{a_1} (x_1x_2)^{a_2}(x_1x_2x_3)^{a_3} \dots (x_1 \dots x_n)^{a_n}$.
The idea of this definition is that the reduced form of $m$
tells us which $d_i$ occur at least once in this $d$-monomial by looking at when the powers change.
For example: the reduced form of $x_1^4x_2^4x_3$ is $x_1^2x_2^2x_3$ and this is the leading monomial of $d_2d_3$.
Another example is $\Red(x_2^2x_3^3)=x_2x_3^2$, which the group element $(x_1,x_3)$ applied to the leading monomial of $d_1d_2$.

We show, in Corollary \ref{lem: Red(d_Ie_I)=lm(e_I)},
that one way to think of $\Red(d)$, for $d$ a $d$-monomial, is to
write out the product of the leading monomials of the $d_i$'s vertically,
then get rid of the repetitions and the $d_n$'s.
For example, let $I = \{ i_1, \dots i_{a} \}, i_1 < i_2 < \dots < i_{a}=n$
and $d=d_{i_1}^{t_1} \dots d_{i_a}^{t_a}$ we may write $\lmlex(d)$ as:
\begin{align*}
\lmlex(d_{i_1})^{t_1} &\left\{
\begin{array}{c}
x_1 \dots x_{i_1} \\
x_1 \dots x_{i_1}
\end{array}
\right.\\
\lmlex(d_{i_2})^{t_2} &\left\{
\begin{array}{c}
x_1 \dots x_{i_1} \dots x_{i_2} \\
x_1 \dots x_{i_1} \dots x_{i_2}
\end{array}
\right.\\
\vdots \;\;\;\;\;\;\;\;\;\;\;&\left\{
\begin{array}{c}
x_1 \dots x_{i_1} \dots x_{i_2} \dots x_{i_j} \\
x_1 \dots x_{i_1} \dots x_{i_2} \dots x_{i_j}
\end{array}
\right.\\
\lmlex(d_{i_{a-1}})^{t_{a-1}} &\left\{
\begin{array}{c}
x_1 \dots x_{i_1} \dots x_{i_2} \dots x_{i_*} \dots x_{i_{a-1}} \\
x_1 \dots x_{i_1} \dots x_{i_2} \dots x_{i_*} \dots x_{i_{a-1}}
\end{array}
\right.\\
\lmlex(d_{n})^{t_n} &\left\{
\begin{array}{c}
x_1 \dots x_{i_1} \dots x_{i_2} \dots x_{i_*} \dots x_{i_{a-1}}\dots x_n \\
x_1 \dots x_{i_1} \dots x_{i_2} \dots x_{i_*} \dots x_{i_{a-1}}\dots x_n.
\end{array}
\right.
\end{align*}

So the reduced form is just:
$$
\begin{array}{ccl}
 && x_1 \dots x_{i_1} \\
 && x_1 \dots x_{i_1} \dots x_{i_2} \\
 && x_1 \dots x_{i_1} \dots x_{i_2} \dots x_{i_*}\\
 && x_1 \dots x_{i_1} \dots x_{i_2} \dots x_{i_*} \dots x_{i_{a-1}},\\
\end{array}
$$
which is clearly $\lmlex(d_{i_1} \dots d_{i_{a-1}})$.

\begin{lemma}\label{lem: red iff i>j}
For $x$-monomials $m=x_1^{m_1} \dots x_n^{m_n}$ and $r=x_1^{r_1} \dots x_n^{r_n}$,
$\Red(m) = \Red(r)$ if and only if we have $(m_i > m_j) \iff (r_i > r_j)$.
\end{lemma}
\begin{proof}
Let $\Red(m) = x_1^{m'_1} \dots x_n^{m'_n}$ and $\Red(r) = x_1^{r'_1} \dots x_n^{r'_n}$.
If $\Red(m) = \Red(r)$ then $m'_i = r'_i$ and $(m_i > m_j) \iff (m'_i > m'_j) \iff (r'_i > r'_j) \iff (r_i > r_j)$.

For the converse:
if $(m_i > m_j) \iff (r_i > r_j)$, then $(m'_i > m'_j) \iff (r'_i > r'_j)$,
and the longest increasing chain of $m'_i$ is the same length as the longest increasing chain of $r'_i$.
Hence $\{ r'_i | i = 1, \dots , n\} = \{ m'_i | i = 1, \dots , n\}$.
\end{proof}

\begin{lemma}\label{lem: red=lmlex}
For an $x$-monomial, $m$,
$\Red(m) \approx e_I'$ for some $I \subseteq [n]$, $n \in I$.
\end{lemma}
\begin{proof}
This follows directly from the definition.
Let $m'$ be a monomial in reduced form,
$m'=x_1^{m'_1}\dots x_n^{m'_n}$ and $\{m'_1, \dots , m'_n \} = \{0 , \dots,  a\}$.
Then there exists a $g \in G$ such that $gm' = m'' = x_1^{m''_1}\dots x_n^{m''_n}$ with
$m''_1 \geq m''_2 \geq \dots \geq m''_n=0$.
We may write:
$$
m'' = (x_1 \dots x_{i_1})^{a} (x_{1+i_2} \dots x_{i_2})^{a-1} \dots (x_{1+i_{a-1}} \dots x_{i_{a}})^1(x_{1+i_{a}} \dots x_n)^0.
$$
But this is equal to:
$\lmlex(d_{i_1}) \lmlex(d_{i_2}) \dots \lmlex(d_{i_{a}})$,
and so: 
$m' \approx m''  = e_{\{i_1, \dots , i_{a},n\} }'$.
\end{proof}

\begin{lemma}\label{lem: x=y => red(x)=red(y)}
For $x$-monomials $x,y \in S$:
$\Red(gx) = g \Red(x)$ and $x \approx y$ implies $\Red(x) \approx \Red(y)$.
cf. \cite[Section 3 Lemma 12]{Kemper}.
\end{lemma}
\begin{proof}
We first show that $\Red(g^{-1}x) = g^{-1} \Red(x)$ for any $g \in G$, this of course shows that $\Red(gx) = g \Red(x)$.
Let $x=x_1^{r_1} \dots x_n^{r_n}$ and $\Red(x) = x_1^{r'_1} \dots x_n^{r'_n}$.
For $\Red(g^{-1}x) = x_1^{s_1} \dots x_n^{s_n}$ and $g^{-1} \Red(x) = x_1^{t_1} \dots x_n^{t_n}$,
we must have that
$\{ r'_i | i = 1, \dots , n\} = \{ s_i | i = 1, \dots , n\} = \{ t_i | i = 1, \dots , n\} = \{0,1, \dots , a-1 \}$,
so by Lemma \ref{lem: red iff i>j} it remains to show that $s_i > s_j$ if and only if $t_i > t_j$.

We defined $G$ as acting on $\{x_1, \dots , x_n \}$,
this gives us an action on $\{1, \dots , n \}$ via $gx_i = x_{g(i)}$.
In this notation $g^{-1}(x_1^{r_1} \dots x_n^{r_n}) = x_1^{r_{g(1)}}\dots x_n^{r_{g(n)}}$.
Hence $t_i = r'_{g(i)}$, so $t_i>t_j$ if and only if $r'_{g(i)} > r'_{g(j)}$
and by the definition of $\Red(x)$ this is if and only if $r_{g(i)} > r_{g(j)}$.
Likewise the definition of $\Red(g^{-1}x)$ states that $s_i > s_j$ if and only if $r_{g(i)} > r_{g(j)}$.
Hence $\Red(gx) = g \Red(x)$ by Lemma \ref{lem: red iff i>j}.

To show that $x \approx y$ implies $\Red(x) \approx \Red(y)$,
note that if $gx=hy$ then $\Red(gx) = \Red(hy)$.
So by the above, $g\Red(x) = h\Red(y)$, which is the same as saying $\Red(x) \approx \Red(y)$.
\end{proof}

\begin{defn}
If $X$ is a set of $x$-monomial such that $x \approx y$ for all $x,y \in X$ (e.g. $X = \Glm (u)$),
then by $\Red(X)$ we mean $\{ \Red(x) | x \in X\}$.
\end{defn}
Note that by Lemma \ref{lem: x=y => red(x)=red(y)},
if $\forall x,y \in X, x \approx y$ then $\forall x',y' \in \Red(X), x' \approx y'$,
so it makes sense to talk about $\Red(X) \approx m$ when $x \approx y$ for all $x,y \in X$.

\begin{lemma}\label{lem: Red(d_iu)=lm(e_I)}
Let $e_I'$ be as in the box and let $u \in S$ be such that $\Red(\lmlex(u)) = e_I'$.
Then for all $t \in I$ we have $\Red(\lmlex(d_t u)) = e_I'$.
\end{lemma}
\begin{proof}
Let $m= \lmlex(u)$ with $m=x_1^{m_1} \dots x_n^{m_n}$ and $e_I' = x_1^{r_1} \dots x_n^{r_n}$.
$\Red(m) = e_I'$ implies $m_i > m_j \iff r_i > r_j$ by Lemma \ref{lem: red iff i>j}.

For $I = \{i_1, \dots , i_a,n\}$ with $1 \leq i_1< i_2 < \dots < i_a < n$,
by definition we have $e_I' = \lmlex(\prod_{j=1}^{a} d_{i_j})$.
By Lemma \ref{lem: Glm}(1) this is equal to
$(x_1 \dots x_{i_1}) (x_1 \dots x_{i_2}) \dots (x_1 \dots x_{i_{a-1}})$.
Collecting all the powers of $x_i$ together we get
\begin{equation}\label{eqn: red}
e_I' = (x_1 \dots x_{i_1})^{a} (x_{i_1 + 1} \dots x_{i_2})^{a-2} \dots (x_{i_{a-2}+1} \dots x_{i_{a}})^1
\end{equation}
So for $i,j \in [n]$ with $i>j$, we have $r_i > r_j$ if and only if $\exists l \in I$ such that $i \geq l > j$.
Hence $m_i > m_j$ if and only if $\exists l \in I$ such that $i \geq l > j$.

By Lemma \ref{lem: Glm}(1) $\lmlex(d_t u) = \lmlex(d_t) \lmlex(u) = \lmlex(d_t) m$.
Let $\lmlex(d_t u) = x_1^{m_1'} \dots x_n^{m_n'}$, then:
$$
x_1^{m_1'} \dots x_n^{m_n'}= \lmlex(d_t u) = (x_1 \dots x_t)m = x_1^{1+m_1} \dots x_t^{1+m_t}x_{t+1}^{m_{t+1}} \dots x_n^{m_n}.
$$
We now compare $(m_i, m_j)$ and $(m_i',m_j')$ for any pair of $i,j \in [n]$.

{For $i,j \leq t$:} we have $m_i' = m_i+1$ and $m_j'= m_j+1$,
so we have $(m_i'>m_j') \iff (m_i > m_j)$.

{For $t < i,j$:} likewise we have $m_i'=m_i$ and $m_j'=m_j$,
so we have $(m_i'>m_j') \iff (m_i > m_j)$.

{For $i \leq t < j$:} we have $m_i' = m_i+1$ and $m_j' = m_j$, so $m_i'> m_j'$.
But, by the observation following Equation \ref{eqn: red} and the fact that $t \in I $, we also have $m_i > m_j$.

Hence  by Lemma \ref{lem: red iff i>j}, $\Red(\lmlex(d_t u)) = \Red(\lmlex(u)) = e_I'$.
\end{proof}

\begin{corol}\label{lem: Red(d_Ie_I)=lm(e_I)}
For $\{n\} \subseteq I \subseteq [n]$, $e_I$ and  $e_I'$ as defined in the box and $d$ a $d_I$-monomial:
$\Red(\lmlex (d )e_I' ) = e_I'$.
\end{corol}
\begin{proof}
By Lemma \ref{lem: Red(d_iu)=lm(e_I)} $\Red( \lmlex( d_t e_I)) = e_I'$ for every $t \in I$.
So by repeated application of this lemma for any $d_I$-monomial, $d$, $\Red( \lmlex( d e_I)) =e_I'$.
By Lemma \ref{lem: Glm}(1) $\lmlex( d e_I) = \lmlex( d )e_I'$.
\end{proof}

\section{Main Theorem}

We now draw together the results of the previous sections to prove that we have a structure theorem.

\begin{lemma} \label{lemma: red}
For $e_I$,$e_I'$ and $V_I$ as in the box, given distinct $d_I$-monomials $r,r_1, \dots , r_m$,
we have $rV_I \cap \left(\sum_{i=1}^m r_iV_I\right) = \{0\}$.
In particular we have $rV_I \cap r'V_I=0$ for $d_I$-monomials $r \neq r'$.

For $u \in V_I-\{0\}$ and $d \in k[d_I]$ we have
$\Red(\Glm(du)) \approx e_I'$.

Conversely, if $m$ is an $x$-monomial
then there exists an $I \subseteq [n]$ with $n \in I$, a $d_I$-monomial, $r$, and a $g \in G$ such that
$\Glm(r g e_I) = \{ m \}$
\end{lemma}

\begin{proof}
In this proof we show that the result holds for a $d_I$-monomial,
then use Lemma \ref{lem: Glm}(6) to get the result about an arbitrary element of $k[d_I]$.

First we make a general observation.
By Lemma \ref{lem: Glm(du) = lm(de_I)}
for $r$ any $d_I$-monomial and any $u,v \in V_I-\{0\}$, we have $\Glm(ru) \approx \Glm(rv) \approx \lmlex(r)e_I'$.
For $r'$ a $d_I$-monomial $r \neq r'$, we have $\lmlex(r)e_I' \neq \lmlex(r')e_I'$.
Hence $\Glm(ru) \not \approx \Glm(r' v)$.

To prove that $rV_I \cap \left(\sum_{i=1}^m r_iV_I\right) = \{0\}$,
let $r,r_1, \dots ,r_m \in R$ be distinct $d_I$-monomials and $u,u_1, \dots , u_m$ be non-zero elements of $V_I$.
Then by the above observation, $\Glm(r_iu_i) \not \approx \Glm(r_j u_j)$ for any $i,j \in [m]$ with $i\neq j$.
So by repeated application of Lemma \ref{lem: Glm}(6),
$\Glm( \sum_{i=1}^m r_iu_i) \approx \Glm(r_j u_j)$ for some $j \in [m]$.
Hence, by the above observation, $\Glm(ru) \not \approx  \Glm(r_ju_j)$ as $r \neq r_j$,
so $\Glm(ru) \not \approx  \Glm(\sum_{i=1}^m r_iu_i)$.
So $rV_I \cap   \left(\sum_{i=1}^m r_iV_I\right) = \{0\}$.
This proves the first statement,
the ``in particular'' statement follows as a special case or from the observation at the start of the proof.

Now we prove that for $d \in k[d_I]$ we have $\Red(\Glm(du)) \approx e_I'$.
For $r$ a $d_I$-monomial, by Lemma \ref{lem: Glm(du) = lm(de_I)}, $\Glm(ru) \approx \lmlex(r)e_I'$.
By Corollary \ref{lem: Red(d_Ie_I)=lm(e_I)}, $\Red(\lmlex(r)e_I') = e_I'$.
So by Lemma \ref{lem: x=y => red(x)=red(y)},  $\Red(\Glm(ru)) \approx e_I'$.
This deals with the case when $d=r$ is a $d_I$-monomial.

For $d = \sum_{i=1}^m \lambda_i r_i$, with $\lambda_i \in k-\{0\}$ and $r_i$ $d_I$-monomials,
the $\Glm(\lambda r_iu)$ are pairwise disjoint.
Hence by Lemma \ref{lem: Glm}(6), there exists an $j\in \{1, \dots, m\}$ such that $\Glm(du) \approx \Glm(r_{j}u)$.
So may prove that $\Red(\Glm(du)) \approx e_I'$ using the ``$d$ is a $d_I$-monomial case'' proved above.

For the converse:
By Lemma \ref{lem: Glm}(7),
it is sufficient to find $I,r$ and $g_m$ for  $m$ with the property that $m \geq_{\text{lex}} gm$ for all $g \in G$.
So for the rest of the proof we assume that $m$ has this property.

Now $m = \prod_{i=1}^n \lmlex(d_i)^{t_i}$, for some $t_i \in \N$.
Let $I = \{ i | t_i \neq 0 \} \cup \{ n \}$, so that $m = d_n^{t_n} \prod_{i\in I} \lmlex(d_i)^{t_i} $.
Then $e_I$ divides $m$ and $m = \lmlex(e_I d_n^{t_n} \prod_{i\in I} \lmlex(d_i)^{t_i-1})$.
Let $r = d_n^{t_n} \prod_{i\in I}d_i^{t_i-1}$
then by Lemma \ref{lem: Glm(du) = lm(de_I)} $\Glm(e_I r) = \{ m\}$.
\end{proof}

\begin{thm}\label{thm: str thm}
Let $e_I$ be elements of $S$ such $\Glm(e_I) = \{ \lmlex(\tilde{d_I})/d_n \}$ and
$\textnormal{stab}_{G}(e_I)= \textnormal{ stab}_G(\lmlex (e_I) )$.
Let $V_I$ be the $kG$-module generated by $e_I$.
Then as $kG$-modules we have:
$$
S \iso \Oplus_{\{n\} \subseteq I \subseteq [n]} k[d_I] \otimes_k V_I
$$
is a structure theorem for $S$,
i.e. the map from right to left is the $kG$-homomorphism $d \otimes_k v \mapsto dv$.
\end{thm}
\begin{proof}
It is clear that the map is a $kG$-map as inclusion and multiplication by $d_i$ are $kG$-maps.
It remains to show that the map is a bijection.

\textbf{Injection:}
By induction on subsets of $[n]$ containing $n$.
The base case is just the observation that for every $I \subseteq [n]$ with $n \in I$, the map $k[d_I]\otimes_k V_I \to S$ is injective.
To see this suppose that $u = \sum_{i=1}^m r_i \otimes u_i \mapsto 0$, where $r_i \neq r_j$ for $i\neq j$ and $u_i \in V_I-\{0\}$.
Then $\sum_{i=1}^m r_i u_i = 0$.
So by the first statement of Lemma \ref{lemma: red},
$r_1u_1 = \sum_{i=2}^m r_i u_i = 0$ and thus $u=0$.
So $k[d_I] \otimes_k V_I \to S$ is an injective map.

For the inductive hypothesis, suppose that given, $A$, a set of subsets of $[n]$ all of which contain $n$
(i.e. $A \subset \mathbb{P}([n])$ and $\forall I \in A, n \in I$),
the map $\Oplus_{I \in A} k[d_I]\otimes_k V_I \to S$ is injective.
We want to show that
$\left(\Oplus_{I \in A} k[d_I]\otimes_k V_I \right) \oplus \left( k[d_J]\otimes_k V_J\right) \to S$
is injective for $J \subseteq [n]$, $n \in J$ and $J \not \in A$.

By Lemma \ref{lemma: red} for all $v \in \phi( k[d_J] \otimes_k V_J ), \Red( \Glm (\phi(u))) \approx e_J'$.
So it is sufficient to show that for $u = \sum_{I \in A} u_I$ with $u_I \in k[d_I] \otimes_k V_I$,
$\Red( \Glm (\phi(u) ) ) \not \approx e_J'$.

By Lemma \ref{lemma: red} $\Red( \Glm (\phi(u_I))) \approx e_I'$,
it is clear that $e_I' \not \approx e_{I'}'$ for $I \neq I'$.
Hence by Lemma \ref{lem: x=y => red(x)=red(y)} the $\Glm (\phi(u_I))$ are disjoint.
By Lemma \ref{lem: Glm}(6), $\Glm (\sum_{I \in A} \phi(u_I)) \approx \Glm (\phi(u_{I_0}))$ for some $I_0 \in A$.
By Lemma \ref{lemma: red} again, $\Red( \Glm (\phi(u) ) ) \approx \Red (\Glm ( \phi(u_{I_0}))) \approx e_{I_0}'$,
and $e_{I_0}' \not \approx e_J'$ as $I_0 \neq J$.
This shows that the map is in injection.

\textbf{Surjection:}
To show that the map is surjective we argue by induction on $\Glm$,
where $\Glm (u) > \Glm (v)$ if $\Glm (u) \succ \Glm (v)$ or if $\Glm (u) \approx \Glm (v)$ and $\Glm (u) \supsetneq \Glm (v)$.

The least $\Glm (u)$ is $\{ 0 \}$, which is clearly mapped onto.
For $u \in S$ assume every $v\in S$ with $\Glm(v)<\Glm(u)$ is mapped onto.
Pick $m \in \Glm (u)$, $\Red(m) = g \cdotp e_I'$ by Lemma \ref{lem: red=lmlex}.
Then by Lemma \ref{lemma: red}, $\exists r \in k[d_I]$ s.t. $\Glm (rge_I)= \{ m \}$,
hence $\Glm(u) > \Glm (u - r g e_I)$.

$r g e_I = \phi (r \otimes_k g e_I)$,
and by the inductive hypothesis $u - r g e_I = \phi(\tilde{u})$ for some $\tilde{u}$,
so $\phi (\tilde{u} + r \otimes_k g e_I) = u$.
\end{proof}

Note that the modules $V_{I}$ are not indecomposable. 
In fact it may be interesting to calculate the vertices of their indecomposable summands as in the example of the upper triangular group,
\cite[above Corollary 9.5]{Symonds mod str 2}, the modules which occur in the structure theorem,
$X_J$ (written as $\bar{X}_{J}(I)$ for $J \subseteq I \subseteq \{1, \dots ,n \}$ in the notation of that paper),
are induced from a subgroup, $U_J$, which depends on the the set of invariants $\{d_j | j \in J \}$.
Be warned that we have adopted different conventions to \cite{Symonds mod str 2},
in particular, for us structure theorems are a sum of $k[d_i | i \in I] \otimes_k X_I$,
but in \cite{Symonds mod str 2} they are a sum of $k[d_i | i \not \in I] \otimes_k X_I$.

It is worth noting that if $k$ were a field, in principal,
we could have shown the the map in Theorem \ref{thm: str thm}
was either injective or surjective and then compared the Hilbert series of the two modules.
However this proved somewhat complicated
as for $I=\{i_1 , \dots , i_m \}$ with $i_1 < \dots < i_m=n$, the dimension over $k$ of $V_I$ is
$\frac{|G|}{|\stab_G(e_I)|} = \frac{n!}{i_1! (i_2-i_1)! \dots (n-i_{m-1})!}$.

\end{document}